\documentclass[draft,a4paper]{amsart}
  \allowdisplaybreaks
\usepackage{mathtools}
\usepackage{enumitem}

\newtheorem{theorem}{Theorem}[section]
\newtheorem{lemma}[theorem]{Lemma}
\newtheorem{corollary}[theorem]{Corollary} 
\newtheorem{proposition}[theorem]{Proposition} 

\theoremstyle{definition}

\theoremstyle{remark}
\newtheorem{remark}[theorem]{Remark}

\numberwithin{equation}{section}

\newcommand*{\indbr}[1]{1_{\{#1\}}}
\DeclareMathOperator*{\EE}{\mathbb{E}} 
\newcommand*{\PP}{\mathbb{P}} 
\newcommand*{\RR}{\mathbb{R}} 
\newcommand*{\CC}{\mathbb{C}} 
\newcommand*{\NN}{\mathbb{N}} 
\newcommand*{\eps}{\varepsilon} 
\newcommand*{\opId}{I} 
\newcommand*{\opHm}[1][m]{H_{#1}} 
   \newcommand*{\opH}{\opHm[0]}
\newcommand*{\opBA}{B} 
\newcommand*{\Lp}[1][p]{L^{#1}([0,\infty))} 
\newcommand*{\Lpseg}[1][p]{L^{#1}([0,1])} 
\newcommand*{\LpCC}[1][p]{L^{#1}(\CC)} 
\newcommand*{\opnorm}[2]{\| #1 \|_{#2\to #2}}
\newcommand*{\Cpml}[1][p,m,\lambda]{C_{#1}} 
   \newcommand*{\Cp}[1][p]{C_{#1}}
\newcommand*{\Dpml}[1][p,m,\lambda]{D_{#1}}
\newcommand*{\cmpl}[1][p,m,\lambda]{c_{#1}}
\newcommand*{\gpm}[1][p,m]{g_{#1}}
   \newcommand*{\gp}[1][p]{\alpha_{#1}}
\newcommand*{\setpm}[1][p,m]{\Omega_{#1}}
\newcommand*{\apml}{\alpha_{0}} 
\newcommand*{\bpml}{\beta_{0}} 

\begin{document}

\title[Beurling-Ahlfors transform on radial functions]{The $L^p$-norms of the Beurling-Ahlfors transform on radial functions}

\author{Micha{\l} Strzelecki}
\address{Institute of Mathematics, University of Warsaw, Banacha 2, 02--097 Warsaw, Poland.}
\email{m.strzelecki@mimuw.edu.pl}


\date{November 24, 2015}


\begin{abstract}
We calculate the norms of the operators connected to the action of the  Beurling-Ahlfors transform on radial function subspaces introduced by Ba\~nuelos and Janakiraman. In particular, we find the norm of the Beurling-Ahlfors transform acting on radial functions  for $p>2$, extending the results obtained by  Ba\~nuelos and Janakiraman, Ba\~nuelos and Os\c{e}kowski, and Volberg for $1<p\leq 2$.
\end{abstract}

\maketitle

\section{Introduction and main results}

The Beurling-Ahlfors transform is a singular operator defined by
\begin{equation*}
\opBA f(z) = -\frac{1}{\pi}\operatorname{p.v.} \int_{\CC} \frac{f(w)}{(z-w)^2} dw,
\end{equation*}
where the integration is with respect to the Lebesgue measure on the complex plane~$\CC$. It plays an important role in the study of quasiconformal mappings and partial differential equations (see e.g. \cite{MR1294669, MR719167}).

A longstanding conjecture of Iwaniec \cite{MR719167} states that for $1<p<\infty$,
\begin{equation*}
\opnorm{\opBA}{\LpCC} = p^*-1,
\end{equation*}
where $p^* = \max\{p, \frac{p}{p-1}\}$. While the lower bound $\opnorm{\opBA}{\LpCC} \geq p^*-1$ was already known to Lehto~\cite{MR0181748}, the question about the opposite estimate remains open. Most results rely on the ideas of Burkholder and the Bellman function technique~\cite{MR1370109,MR2068982,MR2001941,MR2164413,MR2386238,MR3047469}, with the current best being $\opnorm{\opBA}{\LpCC} \leq 1.575 (p^*-1)$  due to Ba\~nuelos and Janakiraman~\cite{MR2386238} (see also~\cite{MR3047469} for an asymptotically better estimate as $p\to\infty$).

However, some sharp results are known for the Beurling-Ahlfors transform restricted to the class of radial functions~\cite{MR2595549,MR2677626,MR3018958,MR3119338,volberg,MR3189475}. In this case we have the representation (see~\cite{MR2595549})
\begin{equation*}
\opBA F(z) = \frac{\bar{z}}{z}\big(\opHm[0] f(|z|^2) - f(|z|^2)\big),
\end{equation*}
where $f:[0,\infty)\to\CC$ is an integrable function, $F(z) = f(|z|^2)$ is the associated radial function, and $\opH$ is the Hardy operator defined by the formula
\begin{equation*}
\opHm[0] f(t) = \frac{1}{t}\int_0^t f(s) ds.
\end{equation*}
Ba\~nuelos and Janakiraman \cite[Theorem~4.1]{MR2595549}  (and later, using other techniques, Ba\~nuelos and   Os\c{e}kowski~\cite[Theorem~5.1]{MR3018958},  Volberg~\cite{volberg}) proved that for $1<p\leq 2$ and any radial function $F\in L^p(\CC)$, we have $\|\opBA F\|_p \leq \frac{1}{p-1} \|F\|_p$. The constant $1/(p-1)$ is the best possible and coincides with the constant from Iwaniec's conjecture. As for $p>2$, Ba\~nuelos and   Os\c{e}kowski~\cite{MR3018958} observed that $\|\opBA F\|_p \leq \frac{2p}{p-1} \|F\|_p$. This bound is asymptotically sharp (and does not agree with the behavior conjectured in the case of all, not only radial, functions).

In their paper, Ba\~{n}uelos and Janakiraman~\cite{MR2595549} went a step further and considered for $m\in\NN$ the operators
\begin{equation*}
(\opId - (1+m)\opHm) f(t) =f(t) -  \frac{1+m}{t^{1+m/2}} \int_0^t f(s) s^{m/2} ds,
\quad f\in L^1_{\text{loc}}([0,\infty)),
\end{equation*}
which correspond to the action of the Beurling-Ahlfors transform on the radial function subspaces
\begin{equation*}
\{ F\in\LpCC : F(re^{i\theta}) = f(r)e^{-im\theta}\}.
\end{equation*}
They proved \cite[Section 5]{MR2595549} that
\begin{equation*}
\opnorm{\opHm}{\Lp} = \frac{1}{m/2 + (p-1)/p}, \quad 1<p<\infty,
\end{equation*}
(with the extremal family $f_{\eps}(t) = t^{-1/p+\eps}\indbr{t\in(0,1)}$), 
   and conjectured  \cite[Conjecture~1]{MR2595549} that the $L^p$-norm of the operator $\opId - (1+m)\opHm$ is equal to  
\begin{equation*}
(1+m)\opnorm{\opHm}{\Lp}  - 1 = \frac{m/2 + 1/p}{m/2 + (p-1)/p}
\end{equation*}
for $1<p<2$. For $p>2$, this number is smaller than one, and cannot be a candidate for the norm of $\opId - (1+m)\opHm$, since the operator $(1+m)\opHm:\Lp\to\Lp$ is not invertible (see Remark~\ref{rem:inv}). 

In fact, the formula
\begin{equation*}
\opHm f(t) = \frac{1}{t^{1+m/2}} \int_0^t f(s) s^{m/2} ds
\end{equation*}
defines a bounded operator on the space $\Lp$ ($1<p<\infty$) not only for natural $m$, but for all $m>-2(p-1)/p$ (see Propositon~\ref{prop:boundness}).
The main goal of this article is to find the $L^p$-norm of the operator $\opId - \lambda\opHm$ for $1<p<\infty$, $m>-2(p-1)/p$, and $\lambda\in\RR$. The case $\lambda=1+m$, $m\in\NN$, corresponds to the action of the Beurling-Ahlfors transform on radial function subspaces considered by Ba\~nuelos and Janakiraman, but it turns out that~\cite[Conjecture~1]{MR2595549} does not hold.

For the formulation of the main result we denote $\gpm = m/2 + (p-1)/p$ for $m>-2(p-1)/p$ and $1<p<\infty$.

\begin{theorem}\label{thm:main}
If $1<p<\infty$, $m>-2(p-1)/p$, and $\lambda\in\RR$, then 
\begin{equation}\label{ineq:main}
\|f-\lambda \opHm f\|_p \leq \Cpml\|f\|_p, \quad f\in\Lp,
\end{equation}
where
\begin{equation*}
\Cpml ^p = \sup\Big\{ \frac{(\beta-\gpm)|\alpha-\lambda|^p + (\gpm-\alpha)|\beta-\lambda|^p}{(\beta-\gpm)|\alpha|^p + (\gpm-\alpha)|\beta|^p} : \ \alpha < \gpm < \beta\Big\}.
\end{equation*}
The inequality is sharp. Moreover, the constant $\Cpml[p,m=0,\lambda=1]^p$ is equal to
\begin{equation}\label{eq:Cp-intr}
\Cp ^p := \sup_{\alpha\leq (p-1)/p} \frac{|\alpha-1|^p}{p(1-\alpha)-1 + |\alpha|^p} = 
  \begin{cases}
  \frac{1}{(p-1)^p} & \text{if } 1<p\leq 2,\\
  \frac{(1+|\alpha_p|)^{p-2}}{p-1} & \text{if } p> 2,
  \end{cases}
\end{equation}
where, for $2<p<\infty$,  $\alpha_p\in\RR$ is the unique negative solution to the equation $(p-1)\alpha_p + 2-p = |\alpha_p|^{p-2}\alpha_p$.
\end{theorem}

\begin{remark}\label{rem:Cmpl-values}
Even for $1<p<2$ (and $\lambda=1+m$), the norm of the operator $\opId - (1+m)\opHm$ is sometimes greater than the conjectured value $(1+m)\gpm^{-1} - 1$. E.g. for $p=3/2$, $m=1$, $\lambda=2$, we have $\Cpml^p\approx 1.81$ (attained in the neighbourhood of $(\alpha,\beta) = (0.4, 5.7)$), whereas $((1+m)\gpm^{-1}-1)^{p} = (7/5)^{3/2} \approx 1.66$.
\end{remark}

\begin{remark}\label{rem:Cmpl-values2}
Apart from the case $m=0$, $\lambda=1$, there are simple formulas for $\Cpml$ if $\lambda\leq 0$ or $p=2$ (see Section~\ref{sec:lambda-pm-proofs}). Moreover, a sufficient and necessary condition for   $\Cpml=|\lambda\gpm^{-1}-1|$ to hold can be formulated (see the proof of Proposition~\ref{prop:majorant-pm} and Section~\ref{sec:lambda-pm-proofs}). Note also that $\Cpml\geq\max\{|\lambda\gpm^{-1}-1|, 1\}$ (see Lemma~\ref{lem:Cpm}). 
\end{remark}

\begin{remark}
Throughout the paper we work with real-valued functions, but Theorem~\ref{thm:main} also holds (with the same constant) for complex-valued functions (see Lemma~\ref{lem:pelcz}).
\end{remark}

The results of Theorem~\ref{thm:main} are new already for $p>2$, $m=0$, and $\lambda=1$, and give immediately the following extension of results obtained by other authors {\cite{MR2595549,MR3018958,volberg}.

\begin{corollary}\label{cor:BA}
For $1<p<\infty$ and any complex-valued radial function $F\in \LpCC$, we have the sharp inequality $\|\opBA F\|_p \leq \Cp \|F\|_p$.
\end{corollary}

The article is organized as follows. A complete and purely analytical proof of inequality~\eqref{ineq:main} is contained  in Section~\ref{sec:lambda-m}. Section~\ref{sec:mart} is designed to show a bigger picture. We  prove a maximal martingale inequality connected to the special case $m=0$ and $\lambda=1$. We also identify the constant $\Cpml[p,0,1]$ and try to explain the main ideas behind the construction of the special functions used in the proofs.

\section{Backstage: the martingale inequality}
  \label{sec:mart}

\subsection{Motivation and results.}

For a martingale $f=(f_n)_{n=0}^{\infty}$ denote its one-sided maximal function by  $f^*_n = \sup_{0\leq j \leq n} f_j$. We also use the notation $f^*_{\infty} = \sup_{0\leq n} f_n$ and $f_{\infty} = \lim_{n\to\infty} f_n$ (if the limit exist a.s.).

Recall that the $L^p$-norm, $1<p<\infty$, of the Hardy operator $\opH$ is equal to $p/(p-1)$. This number is also the best constant in Doob's inequality: for a martingale $(f_k)_{k=0}^n$  we have $\| f_n^*\|_p\leq \frac{p}{p-1} \|f_n\|_p$.
It turns out that the martingale inequality can be used to derive the estimate $\| \opH f\|_p\leq \frac{p}{p-1}\|f\|_p$ for nonnegative and nonincreasing functions \cite{MR577984}; a simple rearrangement argument gives then $\|\opH f\|_p\leq \frac{p}{p-1}\|f\|_p$ for all real-valued $f\in\Lp$.

We consider the following maximal inequality. 

\begin{theorem}\label{thm:mart}
For any martingale $(f_n)_{n=0}^{\infty}$, we have
\begin{equation}\label{ineq:mart}
\|f_n-f_n^*\|_p \leq \Cp \|f_n\|_p, \quad 1<p<\infty, n\geq 0,
\end{equation}
where $\Cp$ is defined in~\eqref{eq:Cp-intr}. The inequality is sharp.
\end{theorem}

The quantity $\|f_n-f_n^*\|_p$ seems natural to study, but the main motivation is the aforementioned link to the Hardy operator (see Section~\ref{sec:proof-mart} for the proof). Note that this approach is different from that of Ba\~nuelos and Os\c{e}kowski \cite{MR3018958}, who used estimates for pure-jump martingales, and the analytical approaches of Ba\~nuelos and Janakiraman \cite{MR2595549}, and Volberg \cite{volberg}.

\begin{corollary}\label{cor:main-monotone}
Let $1<p<\infty$. If $f\in\Lp$ is real-valued and nonincreasing, then
\begin{equation*}
\|f-\opH f\|_p \leq \Cp\|f\|_p.
\end{equation*}
\end{corollary}

Quite unexpectedly, some difficulties arise at the stage of rearrangements. In our setting it is possible that
\begin{equation*}
\|g- \opH g\|_p < \| f- \opH f \|_p,
\end{equation*}
where $g$ denotes the nonincreasing rearrangement of a real-valued function $f\in\Lp$ (examples can be found with $f$ being a (positive) step function, in which case $\opH g$, $\|g- \opH g\|_p$, $\| f- \opH f \|_p$ can be explicitly calculated). Hence, it seems that Corollary~\ref{cor:main-monotone} does not directly imply Theorem~\ref{thm:main} (for $m=0$, $\lambda=1$).  Fortunately, it is possible to use the tools from the proof of the martingale inequality (and adapt them to work not only for $m=0$, but for $m>-2(p-1)/p$ and all $\lambda\in\RR$) to obtain our main result (see Sections~\ref{sec:proof-main} and~\ref{sec:lambda-m}).

\subsection{Method of the proof of Theorem~\ref{thm:mart} and a lower bound for the best constant.}
  \label{subsec:lower-bound}

We follow Burkholder's approach to the Doob inequality \cite[p. 578]{MR1108183}: in order to prove inequality~\eqref{ineq:mart}, it suffices to find an appropriate special function (for further reading about maximal martingale inequalities see also~\cite[Chapter~7]{MR2964297}).

\begin{proposition}\label{prop:method}
Let $V(x,y) = |x-y|^p - C^p|x|^p$ and suppose that $U:\RR^2\to\RR$ satisfies the following conditions.
\begin{enumerate}[label=\upshape\arabic*.]
	\item (Majorization) If  $x\leq y$, then $
		V(x,y) \leq U(x,y)$.
	\item (Initial condition) For all $x\in\RR$, we have $U(x,x) \leq 0$.
	\item (Maximal condition) If  $ x\leq y, h\in\RR$, then
		\begin{equation*}
		U(x+h, (x+h)\lor y)	\leq U(x+h,y).
		\end{equation*}
\item (Concavity) For all $y\in\RR$, the function $U(\cdot,y):\RR\to\RR$ is concave. \end{enumerate}
Then $\|f_n-f_n^*\|_p \leq C\|f_n\|_p$ for any martingale $(f_n)_{n= 0}^{\infty}$ and any $n\geq 0$.
\end{proposition}

\begin{proof}
It suffices to consider the inequality  for simple martingales (in which case all expressions below are integrable).
Conditions 3 and 4 imply that
\begin{align*}
\EE U(f_n,f^*_n) &= \EE U\big(f_{n-1} + (f_n-f_{n-1}), (f_{n-1} + (f_n-f_{n-1}))\lor f^*_{n-1}\big)\\
&\leq \EE U(f_{n-1} + (f_n-f_{n-1}), f^*_{n-1})\\
&\leq \EE U(f_{n-1}, f^*_{n-1}) + \EE (f_n-f_{n-1}) U_{x^+}(f_{n-1}, f^*_{n-1}),
\end{align*}
where $U_{x^+}$ denotes the right derivative.
Moreover, $\EE (f_n-f_{n-1}) U_{x^+}(f_{n-1}, f^*_{n-1})=0$ because $f$ is a martingale. Hence, $\EE U(f_n,f^*_n)\leq \EE U(f_{n-1}, f^*_{n-1})$. Thus, using Conditions 1 and 2, we arrive at
\begin{multline*}
\| f_n -  f^*_n \|_p^p -C^p \|f_n\|_p^p = \EE V(f_n,f^*_n)\\
 \leq \EE U(f_n,f^*_n)
\leq \ldots\leq  \EE U(f_0, f^*_0) =  \EE U(f_0, f_0)\leq 0.
\end{multline*}
This ends the proof.
\end{proof}

\begin{remark} In the above proof it is enough to have $\EE (f_n-f_{n-1}) U_{x^+}(f_{n-1}, f^*_{n-1})\leq 0$. This inequality holds if $f$ is a nonnegative submartingale and $U_{x^+}(x,y)\leq 0$ for $y\geq 0$. This additional assumption is satisfied by the function $U$ which we construct in Section~\ref{sec:majorant}. In particular, for any martingale $(f_n)_{n=0}^{\infty}$ also
\begin{equation*}
\big\||f_n| - \sup_{0\leq j \leq n} |f_j|\big\|_p \leq \Cp \|f_n\|_p
\end{equation*}
holds, since $(|f_n|)_{n=0}^{\infty}$ is a nonnegative submartingale whenever  $(f_n)_{n=0}^{\infty}$ is a martingale. This bound is sharp in the case $1<p\leq 2$ (see the example in Section~\ref{sec:proof-mart}), but the constant $\Cp$ does not seem to be the best possible for $p>2$.
\end{remark}

There is an abstract way for finding a candidate for the function from Proposition~\ref{prop:method}. Namely, let $V(x,y) = |x-y|^p - C^p|x|^p$ and define
\begin{equation*}
U^0(x,y) = \sup \{ \EE V(f_{\infty}, f^*_{\infty}\lor y ) : f_0=x\},
\end{equation*}
where the supremum is taken over the class $\mathcal{M}$ consisting of all simple martingales $f=(f_n)_{n=0}^{\infty}$ on the probability space $[0,1]$ equipped with the Borel $\sigma$-algebra and the Lebesgue measure (the filtration may vary). This approach has one main drawback: the expression defining $U^0$ is hard to work with. Nonetheless, we can use the function $U^0$ to extract important information: a lower bound for the constant $C=C(p)$, with which the martingale inequality holds (and on which the function $V$ depends). For explicit examples of extremal martingales see Section~\ref{sec:proof-mart}.

\begin{proof}[Sharpness of \eqref{ineq:mart}] let $1<p<\infty$ be fixed.
First note that by the triangle and Doob's inequality the estimate~\eqref{ineq:mart}  holds with some finite constant. Let us denote it by $C$ (of course it may depend on $p$) and let $V$, $U^0$ be the functions defined in the preceding paragraph. Note that, as for now, we do not claim that $U^0<+\infty$. 

Clearly, $U^0(x,y)\geq V(x,x\lor y)$ (since a constant martingale, $f_n\equiv x$, belongs to $\mathcal{M}$), $U^0(x,y)= U^0(x,x\lor y)$ (since $f_0\leq f_{\infty}^*$), and $U^0(ax,ay) = |a|^p U^0(x,y)$. A~``splicing'' argument (cf. \cite{MR2964297}) gives us concavity of $U^0(\cdot,y)$: if $\lambda\in(0,1)$, $f,g\in \mathcal{M}$, $f_0=x_1$, and $g_0=x_2$, then the process defined by  $h_0=\lambda x_1 + (1-\lambda) x_2$ and
\begin{equation*}
h_n(\omega) =
	\lambda f_{n-1}(\omega/\lambda) \indbr{\omega\in [0,\lambda)} + (1- \lambda) g_{n-1}\big((\omega-\lambda)/(1-\lambda)\big) \indbr{\omega\in [\lambda,1)}, \quad  n\geq 1,
\end{equation*}
is a simple martingale starting from $x=\lambda x_1 + (1-\lambda)x_2$. Hence
\begin{equation*}
U^0(x,y) \geq \EE V(h_{\infty},h_{\infty}^*\lor y) = \lambda \EE V(f_{\infty},f_{\infty}^*\lor y) + (1-\lambda)\EE V(g_{\infty},g_{\infty}^*\lor y),
\end{equation*}
which after taking the suprema over $f$ and $g$ yields the claim.

Moreover, if $f\in\mathcal{M}$ satisfies $f_0=y$, then $\EE V(f_{\infty}, f_{\infty}^*\lor y) = \EE V(f_{\infty}, f_{\infty}^*)\leq 0$, where the inequality follows from the assumption that the martingale inequality is satisfied with constant $C$. Therefore $U^0(y,y)\leq 0$ for all $y\in\RR$, and hence $U^0(x,y)<+\infty$ for any $x,y\in\RR$. Indeed, $U^0$ is concave with respect to the first variable, and a concave function on the  real line, which takes values in the set $(-\infty,+\infty]$, and is equal to $+\infty$ at some point, is identically equal to $+\infty$.

We now exploit the function $U^0$ to get an estimate of the constant $C$.  Fix $\alpha \leq (p-1)/p$ and $\delta, t\in(0,1)$. The properties of $U^0$ imply
\begin{align*}
U^0(1,1)
&\geq \frac{\delta}{1-\alpha+\delta} U^0(\alpha,1) + \frac{1-\alpha}{1-\alpha+\delta} U^0(1+\delta,1) \\
&\geq  \frac{\delta}{1-\alpha+\delta} V(\alpha,1)  + \frac{1-\alpha}{1-\alpha+\delta} U^0(1+\delta,1+\delta)\\
&\geq\frac{\delta}{1-\alpha+\delta} V(\alpha,1)\\
\MoveEqLeft[-2] +(1-t)\frac{1-\alpha}{1-\alpha+\delta} (1+\delta)^p U^0(1,1)
+ t\frac{1-\alpha}{1-\alpha+\delta} (1+\delta)^p V(1,1),
\end{align*}
which can be rewritten in the form
\begin{equation*}
U^0(1,1) \frac{1-\alpha+\delta - (1-t)(1-\alpha)(1+\delta)^p }{\delta}
\geq V(\alpha,1) + t\frac{1-\alpha}{\delta} (1+\delta)^p V(1,1).
\end{equation*}
Now, for $\delta<1/p$, we put $t=\delta(p-1/(1-\alpha))$ (note that $t\in[0,1]$, since $\alpha\leq (p-1)/p$), take $\delta\to 0^+$, and arrive at
\begin{equation*}
0\geq V(\alpha,1)+(p(1-\alpha)-1) V(1,1).
\end{equation*}
Using the definition of the function $V$ we can solve this inequality with respect to $C$. Taking the supremum over $\alpha\leq (p-1)/p$ yields then
\begin{equation*}
C^p \geq \sup_{\alpha\leq (p-1)/p}\frac{|\alpha-1|^p}{p(1-\alpha)-1 + |\alpha|^p}
\end{equation*}
(note that $p(1-\alpha)-1+|\alpha|^p$ is strictly positive for $\alpha\neq 1$, since the function $\alpha\mapsto p\alpha-p+1$ is tangent to the convex function $\alpha\mapsto |\alpha|^p$ at $\alpha=1$).
Hence the best constant with which the martingale inequality is satisfied is not smaller than the right-hand side of the above inequality.
\end{proof}

\subsection{Finding the concave majorant.}

In this subsection we give some informal reasoning, which is helpful in guessing an explicit formula for the function $U$ satisfying the assumptions of Proposition~\ref{prop:method}. As before, we denote $V(x,y)=|x-y|^p-C^p|x|^p$. We look for a function $U$ such that $U(\cdot, y)$ is not only concave, but even affine: let
\begin{equation*}
U(x,y) = p(|\alpha y - y|^{p-2} (\alpha y - y) - C^p|\alpha y|^{p-2}\alpha y)(x-\alpha y) + |\alpha y - y|^{p} - C^p|\alpha y|^{p}
\end{equation*}
be the tangent to $V(\cdot,y)$ at the point $x=\alpha y$ (for some $\alpha$, yet to be determined). Note that if $V$ was concave with respect to the first variable, then such a choice of $U$ would automatically guarantee the majorization property (i.e. $V(x,y)\leq U(x,y)$). Unfortunately, this is not the case in our setting (cf.~Lemma~\ref{lem:3con-pm}).

The maximal condition states that $U(x+h, x+h)\leq U(x+h,y)$ for $x+h>y$, and implies $U_y(x,x) \leq 0$. Let us assume that  $U_y(x,x) = 0$. Some calculations reveal that this condition is equivalent  to 
\begin{equation*}
C^p  = |\alpha^{-1}-1|^{p-2}(\alpha^{-1} - 1)\big(\big((p-1)(1-\alpha)\big)^{-1}  - 1\big)
\end{equation*}
(provided that $\alpha\notin\{ 0, 1\}$). Taking such $C$ we arrive at
\begin{equation*}
U(x,y) = -    \frac{|1 - \alpha|^{p-2}}{p-1}  |y|^{p-2}y(px-y(p-1)).
\end{equation*}
Note that for this choice  the initial condition (i.e. $U(x,x)\leq0$) is also satisfied.

Moreover, the preceding subsection suggests that for the right choice of $\alpha$ we should have 
\begin{multline*}
|\alpha^{-1}-1|^{p-2}(\alpha^{-1} - 1)\big(\big((p-1)(1-\alpha)\big)^{-1}  - 1\big)
  = \sup_{\alpha'\leq (p-1)/p}\frac{|\alpha'-1|^p}{p(1-\alpha')-1 + |\alpha'|^p}.
\end{multline*}
We identify the correct values of $\alpha=\alpha(p)$ and $C=C(p)$ in some technical lemmas in the next section. This is relatively easy in the case $1<p\leq2$, where one can simply take $\alpha(p) = (p-1)/p$.

In Section~\ref{sec:majorant}, we check that for these choices the function $U$ is indeed the majorant of $V$. We prove the martingale inequality~\eqref{ineq:mart} in Section~\ref{sec:proof-mart}.

\subsection{Technical lemmas}
  \label{sec:tech-lem}

The first three results are needed to identify the value of the optimal constant in the martingale inequality~\eqref{ineq:mart}.

\begin{lemma}\label{lem:alpha-a}
For each $p\in(2,\infty)$, there exists exactly one number $\alpha_p\leq (p-1)/p$ such that
\begin{equation*}
 \frac{|\alpha_p-1|^p}{p(1-\alpha_p)-1 + |\alpha_p|^p} = \sup_{\alpha\leq (p-1)/p} \frac{|\alpha-1|^p}{p(1-\alpha)-1 + |\alpha|^p}=\sup_{\alpha\neq 1} \frac{|\alpha-1|^p}{p(1-\alpha)-1 + |\alpha|^p}.
\end{equation*}
Moreover, $(p-1)\alpha_p+ 2-p = |\alpha_p|^{p-2}\alpha_p$ and $\alpha_p < -(p-1)^{1/(p-2)}<0$.
\end{lemma}

\begin{proof}
Recall that $p(1-\alpha)-1+|\alpha|^p$ is strictly positive for $\alpha\neq 1$ (since the function $\alpha\mapsto p\alpha-p+1$ is tangent to the convex function $\alpha\mapsto |\alpha|^p$ at $\alpha=1$). Moreover, $\lim_{\alpha\to 1} |\alpha -1|^p/(p(1-\alpha)-1+|\alpha|^p) = 0$. Hence the function
\begin{equation*}
h(\alpha) =\frac{|\alpha-1|^p}{p(1-\alpha)-1 + |\alpha|^p}\indbr{\alpha\neq 1}
\end{equation*}
is continuous. Its derivative (for $\alpha\neq 1$) is equal to
\begin{equation*}
h'(\alpha)=\frac{ p|\alpha-1|^{p-2}(\alpha-1)(p(1-\alpha)-1 + |\alpha|^p) - |\alpha-1|^p (-p +p |\alpha|^{p-2}\alpha)}{(p(1-\alpha)-1 + |\alpha|^p)^2},
\end{equation*}
which is nonpositive if and only if
\begin{equation*}
(\alpha-1)\big((p(1-\alpha)-1 + |\alpha|^p) -(\alpha-1) (-1 + |\alpha|^{p-2}\alpha)\big)\leq 0,
\end{equation*}
which we can simplify to
\begin{equation}\label{eq:alpha-az}
(\alpha-1)\big(|\alpha|^{p-2}\alpha -(p-1)\alpha - 2+p\big)\leq 0
\end{equation}

The function $\alpha\mapsto(p-1)\alpha + 2-p$ is linear and tangent (at $\alpha=1$) to the function $\alpha\mapsto |\alpha|^{p-2}\alpha$, which is strictly concave on $(-\infty,0]$ and strictly convex on $[0,\infty)$. Therefore, the equation $(p-1)\alpha + 2-p = |\alpha|^{p-2}\alpha$ has exactly one negative solution, which we denote by $\alpha_p$. Moreover, the  inequality~\eqref{eq:alpha-az} holds if and only if $\alpha\in[\alpha_p,1]$. Hence, the function $h$ is increasing on $(-\infty, \alpha_p]$, decreasing on $[\alpha_p,1]$, and increasing on $[1,\infty)$. The observation that
$\lim_{\alpha\to\pm \infty} h(\alpha)=1$ ends the proof of the first part of the lemma.

Moreover, the inequality~\eqref{eq:alpha-az} does hold for $\alpha= -(p-1)^{1/(p-2)}$ and hence $\alpha_p < -(p-1)^{1/(p-2)}$.
\end{proof}

\begin{lemma}\label{lem:alpha-b}
Let $\alpha_p$ be the number defined for $p>2$ in Lemma~\ref{lem:alpha-a}. Then
\begin{multline*}
|\alpha_p^{-1}-1|^{p-2}(\alpha_p^{-1}-1)\big(\big((p-1)(1-\alpha_p)\big)^{-1}-1\big)\\
= \frac{|\alpha_p-1|^p}{p(1-\alpha_p)-1 + |\alpha_p|^p} = \frac{(1+|\alpha_p|)^{p-2}}{p-1} >1.
\end{multline*}
\end{lemma}

\begin{proof}
We have
\begin{equation*}
\big((p-1)(1-\alpha_p)\big)^{-1} - 1 = \frac{(p-1)\alpha_p+ 2-p}{(p-1)(1-\alpha_p)} = \frac{|\alpha_p|^{p-2}\alpha_p}{(p-1)(1-\alpha_p)}
\end{equation*}
and hence the first and third expressions are equal.
Also
\begin{equation*}
p(1-\alpha_p)-1 + |\alpha_p|^p = p(1-\alpha_p)-1 +\alpha_p((p-1)\alpha_p+2-p) = (p-1)(\alpha_p-1)^2
\end{equation*}
and hence the second and third expressions are equal. Finally, the inequality follows directly from the estimate for $\alpha_p$ from the preceding lemma.
\end{proof}

\begin{lemma}\label{lem:alpha-c} Let $p\in(1,2)$. If we denote $\alpha_p = (p-1)/p$, then
\begin{multline*}
\frac{1}{(p-1)^p} = \frac{|\alpha_p-1|^p}{p(1-\alpha_p)-1 + |\alpha_p|^p} = \sup_{\alpha\leq (p-1)/p} \frac{|\alpha-1|^p}{p(1-\alpha)-1 + |\alpha|^p} \\
= |\alpha_p^{-1}-1|^{p-2}(\alpha_p^{-1}-1)\big(\big((p-1)(1-\alpha_p)\big)^{-1}-1\big).
\end{multline*}
\end{lemma}

The proof is less involved than in the case $p>2$. Therefore we leave the details of checking that the function
\begin{equation*}
\alpha \mapsto \frac{|\alpha-1|^p}{p(1-\alpha)-1 + |\alpha|^p}, \quad \alpha\in(-\infty,(p-1)/p],
\end{equation*}
attains its maximum at $\alpha=(p-1)/p$ to the Readers. Let us only remark that in contrast to the case $p>2$, for $1<p<2$ we have
\begin{equation*}
 \sup_{\alpha\neq 1} \frac{|\alpha-1|^p}{p(1-\alpha)-1 + |\alpha|^p} = \infty.
\end{equation*}

We also need the following technical lemma.

\begin{lemma}\label{lem:zzz}
If $p\in(1,2)$, then
\begin{align*}
p^{p-2} &\geq (p-1)^{p-1},\\
(p+1)^{p-1}&\geq (2p-1)(p-1)^{p-1}.
\end{align*}
\end{lemma}

\begin{proof}
Both  inequalities are satisfied in the limit for $p\to 1^+$ and $p\to2^-$. To prove the first, we notice that the difference of the logarithms of both sides is a concave function since
\begin{equation*}
\big( (p-2)\ln (p)- (p-1)\ln (p-1)\big)'' = \frac{p-2}{(p-1)p^2}\leq 0,
\end{equation*}
for $p\in(1,2)$.

In order to prove the second inequality, we substitute  $s=p-1$, divide both sides by $(2s+1)s^s$, take the logarithm of both sides, and arrive at the following equivalent formulation of the assertion:
\begin{equation*}
s\ln (1+2/s)- \ln (2s+1) \geq 0, \quad s\in(0,1).
\end{equation*}
The left-hand side is a concave function, since  for $s\in(0,1)$,
\begin{equation*}
\big( s\ln (1+2/s)- \ln (2s+1)\big)'' = \frac{4(s^3-1)}{s(s+2)^2(2s+1)^2}\leq 0.
\end{equation*}
Hence the assertion of the lemma holds.
\end{proof}

\subsection{The special function}
  \label{sec:majorant}

Define the constant $\Cp$ by the formula
\begin{equation*}\label{eq:constant-intro}
\Cp ^p= \sup_{\alpha\leq (p-1)/p} \frac{|\alpha-1|^p}{p(1-\alpha)-1 + |\alpha|^p} = 
  \begin{cases}
  \frac{1}{(p-1)^p} & \text{if } 1<p\leq 2,\\
  \frac{(1+|\alpha_p|)^{p-2}}{p-1} & \text{if } p> 2.
  \end{cases}
\end{equation*}
Here $\alpha_p\in\RR$, $2<p<\infty$ is the unique negative solution to the equation $(p-1)\alpha_p + 2-p = |\alpha_p|^{p-2}\alpha_p$ (see~Lemma~\ref{lem:alpha-a}). We also denote $\alpha_p=(p-1)/p$ for $1<p\leq 2$.

We introduce the special functions
\begin{align*}
V(x,y) &= |x-y|^p - \Cp^p|x|^p,\\
U(x,y) &=  -  \frac{|1-\alpha_p|^{p-2}}{p-1} |y|^{p-2}y(px-(p-1)y)\\
&=
  \begin{cases}
   - \frac{1}{(p-1)p^{p-2}} |y|^{p-2}y(px-(p-1)y)  & \text{if } 1<p \leq 2,\\
   -  \Cp^p |y|^{p-2}y(px-(p-1)y)  & \text{if } p>2.
  \end{cases}
\end{align*}
The following proposition is the core of the proof of the martingale inequality and the main result (in the case $m=0$, $\lambda=1$). Note that the assertion is stronger than  the majorization condition from Proposition~\ref{prop:method}.

\begin{proposition}\label{prop:majorant}
For $1<p<\infty$ and any $x,y\in\RR$, we have $V(x,y)\leq U(x,y)$.
\end{proposition}

\begin{proof}[Proof for $2<p<\infty$]
The inequality is satisfied for $y=0$ and $x=y$, so by homogeneity it is enough to consider it for $y=1$ and $x\neq 1$. We can rewrite it as
\begin{equation*}
\Cp^p(|x|^p-px + p-1) \geq |x-1|^p. 
\end{equation*}
For $x\neq1$ the left-hand side is positive (the function $x\mapsto px-p+1$ is tangent to the convex function $x\mapsto |x|^p$ at $x=1$), and therefore we conclude that the assertion is equivalent to
\begin{equation*}
\Cp^p \geq \sup_{x\neq 1}\frac{|x-1|^p}{p(1-x)-1 + |x|^p},
\end{equation*}
which is true by Lemma~\ref{lem:alpha-a}.
\end{proof}

\begin{remark} The above proof stresses the fact that $\Cp$ is chosen exactly so, that the statement is true, but we can also use a slightly different approach.
Again, it is enough to consider $y=1$.
The function $V(\cdot,1)$ is continuously differentiable, its second derivative exists in all but two points, and moreover $V_{xx}(x,1) = 0$ if and only if $|x-1|^{p-2} = \Cp^p|x|^{p-2}$ or equivalently $x = 1/(1\pm \Cp^{p/(p-2)})$. Hence the function $V(\cdot,1)$ is concave on the interval $(-\infty, a]$, convex on the interval $[a,b]$, and again concave on the interval $[b, \infty)$, where $a=1/(1- \Cp^{p/(p-2)})$, $b=1/(1+\Cp^{p/(p-2)})$. Moreover $U(\cdot,1)$ is the tangent to $V(\cdot, 1)$ at the points $\alpha_p$ and $1$. This implies the inequality $V(x,1)\leq U(x,1)$ for $x\in\RR$ (see Lemma~\ref{lem:3con-pm} below).
\end{remark}

We turn to the case $1<p<2$ (for $p=2$ the assertion is trivial). The argument is similar to that above, but slightly more complicated.

\begin{proof}[Proof for $1<p<2$.]
The inequality is satisfied for $y=0$, so by homogeneity it is enough to consider it for $y=1$. We can rewrite it as (recall that $\gp = (p-1)/p$)
\begin{equation*}
\Cp^p\Big(|x|^p-p\frac{\gp^{p-1}}{1-\gp}(x - \gp) \Big) \geq |x-1|^p. 
\end{equation*}
It is easy to see that left-hand side is strictly positive (the global minimum is attained for $x = \gp(1-\gp)^{-1/(p-1)}$, for which the expression in the brackets on the left-hand side is equal to
\begin{equation*}
|\gp|^p (1-\gp)^{-p/(p-1)}\big(1 - p  + p(1-\gp)^{1/(p-1)} \big),
\end{equation*}
which is positive by the first inequality from Lemma~\ref{lem:zzz}).
Therefore we conclude that the assertion is equivalent to the inequality
\begin{equation}\label{ineq:sup}
\Cp^p \geq  \frac{|x-1|^p}{|x|^p-p\frac{\gp^{p-1}}{1-\gp}(x - \gp)}
\end{equation}
holding for every $x\in\RR$. We denote the right-hand side of the above inequality by $R(x)$.
A calculation shows that $R'(x)$ is positive if and only if 
\begin{multline*}
p|x-1|^{p-2}(x-1)\Big(|x|^p-p\frac{\gp^{p-1}}{1-\gp}(x - \gp)\Big) - |x-1|^{p}\Big(p|x|^{p-2}x-p\frac{\gp^{p-1}}{1-\gp}\Big) \\
= p|x-1|^{p-2}(x-1)\big(T(x) - S(x)\big)\geq 0,
\end{multline*}
where we have denoted
\begin{align*}
S(x) &= \frac{\gp^{p-1}}{1-\gp}\big( (p-1)x -p\gp+1 \big) = p\gp^{p-1}\big( (p-1)x -p+2 \big),\\
T(x) &= |x|^{p-2}x.
\end{align*}
The function $T$ is convex on $(-\infty,0)$ and concave on $(0, \infty)$, since $1<p<2$. Moreover, $S(\gp) = T(\gp) $ and
\begin{equation*}
S'(\gp) =  p(p-1)\gp^{p-1} < (p-1)\gp^{p-2} = T'(\gp).
\end{equation*}
We conclude that the equation $S(x)=T(x)$ has three solutions: $x_1<0$, $x_2=\gp$, and $x_3>\gp$. Moreover, $x_3\geq (p+1)/p>1$ since $S((p+1)/p) \leq T((p+1)/p)$ by the second inequality from Lemma~\ref{lem:zzz}.

Therefore, the function $R$ is decreasing on each of the intervals $(-\infty, x_1)$,  $(\gp, 1)$, $(x_3,\infty)$, and  increasing on $(x_1,\gp)$ and $(1,x_3)$.
Since $R(\gp) = (1/\gp-1)^p = \Cp^p$ and $\lim_{x\to-\infty} R(x)=1\leq \Cp^p$, in order to prove \eqref{ineq:sup} it is enough to check that $\Cp^p \geq R(x_3)$. But $S(x_3) = T(x_3) = x_3^{p-1}$, so
\begin{multline*}
|x_3|^p- p\frac{\gp^{p-1}}{1-\gp}(x - \gp) = x_3 S(x_3)-p\frac{\gp^{p-1}}{1-\gp}(x - \gp)= \frac{\gp^{p-1}}{1-\gp} (p-1)(x_3-1)^2
\end{multline*}
and consequently
\begin{equation*}
R(x_3) =  \frac{1-\gp}{\gp^{p-1}} \cdot\frac{(x_3-1)^{p-2}}{(p-1)} = \frac{p^{p-2}}{(p-1)^{p}} (x_3-1)^{p-2}.
\end{equation*}
Hence, $\Cp^p \geq R(x_3)$ is equivalent to $1\leq p(x_3-1)$ (recall that $p-2<0$).
Since we already know that $x_3\geq (p+1)/p$, the proof is finished.
\end{proof}

\subsection{Proof of the martingale inequality}
  \label{sec:proof-mart}

In order to prove the martingale inequality, we just gather the results of the preceding sections.

\begin{proof}[Proof of inequality~\eqref{ineq:mart}]
For $p\in(1,\infty)$, the functions $V$ and $U$ defined in Section~\ref{sec:majorant} satisfy all assumptions of Proposition~\ref{prop:method}. Indeed, the majorization property follows from Proposition~\ref{prop:majorant}. The initial condition is satisfied, since $U(x,x) = -|x|^p|1-\alpha_p|^{p-2}/(p-1)\leq 0$. If $y\leq 0\leq x+h$, then $U(x+h,x+h)\leq 0\leq U(x+h,y)$. If on the other hand $y<x+h<0$ or $0<y<x+h$, then there exists $\xi\in(y,x+h)$ such that
\begin{multline*}
U(x+h,x+h) - U(x+h,y) = U_y(x+h,\xi)(x+h-y) \\
= p|1-\alpha_p|^{p-2}|\xi|^{p-2}(\xi-x-h)(x+h-y)\leq 0.
\end{multline*}
This implies the maximal condition. Finally, $U$ is clearly concave with respect to the first variable. Hence the martingale inequality~\eqref{ineq:mart} holds by Proposition~\ref{prop:method}.
\end{proof}

Moreover, from the abstract argument in Subsection~\ref{subsec:lower-bound} we already know that the constant $\Cp$ is optimal.  Let us however give explicit extremal examples here.

\begin{proof}[Sharpness of~\eqref{ineq:mart}] Fix $1<p<\infty,$ $\alpha \in(-\infty, (p-1)/p)\setminus\{0\}$ and $s\in(0,1)$, and let $\beta=\beta(s)$ be given by the relation $s \alpha + (1-s)\beta=1$. Observe that if $s$ is sufficiently small (depending on $p,\alpha$), then $(1-s)\beta^p>1$ (indeed,
the inequality $(1-s\alpha)^{p} > (1-s)^{p-1}$ can be verified by comparison of derivatives at $s=0$). We consider a martingale $(f_n)_{n=0}^{\infty}$ such that $f_0 = 1$, and such that conditioned on the event $\{(f_n,f_n^*)=(x,x)\}$, one of the following events occurs:
\begin{enumerate}
\item With probability $s$ we have $f_{n+1} = \alpha x$ and the martingale stops, i.e. $f_{n+1}=f_{n+2}=\ldots$; note that in this case $f_{n+1}^* = f_n = x$.
\item With probability $1-s$, we have $f_{n+1} = \beta x$ and the evolution continues according to our rules. In this case $f_{n+1}^*=f_{n+1}$.
\end{enumerate}
Note that $f_n$ takes values in the set $\{\alpha, \alpha\beta, \alpha\beta^2, \ldots, \alpha\beta^{n-1},\beta^n \}$. Moreover, $\PP(f_n=\alpha\beta^k) = s(1-s)^k$ for $k\in\{0,\ldots, n-1\}$, $\PP(f_n=\beta^n) = (1-s)^n$, and
\begin{equation*}
\EE |f_n|^p\indbr{f_n\neq \beta^n} = \sum_{k=0}^{n-1} s(1-s)^k (|\alpha|\beta^k)^p = s|\alpha|^p\frac{1-(1-s)^n\beta^{np}}{1-(1-s)\beta^p}.
\end{equation*}
Also, $f_n-f_n^* = (1-1/\alpha)f_n$ if $f_n\neq \beta^n$, and $f_n-f_n^*=0$ if $f_n = \beta^n$. Hence, the $p$-th power of the constant with which the martingale inequality~\eqref{ineq:mart} holds has to be equal at least
\begin{multline*}
\lim_{s\to 0^+} \lim_{n\to\infty} \frac{\| f_n - f_n^*\|_p^p}{\|f_n\|_p^p} = \lim_{s\to 0^+} \lim_{n\to\infty} \frac{|1-1/\alpha|^p s|\alpha|^p\frac{1-(1-s)^n\beta^{np}}{1-(1-s)\beta^p}}{s|\alpha|^p\frac{1-(1-s)^n\beta^{np}}{1-(1-s)\beta^p} + (1-s)^n \beta^{np}} \\
= \lim_{s\to 0^+} \frac{|\alpha-1|^p s}{s|\alpha|^p + (1-s)\beta^p-1} = \frac{|\alpha-1|^p}{|\alpha|^p -p\alpha + p-1},
\end{multline*}
where we used $(1-s)\beta^p >1$ in the second last equality, and $\beta=(1-s\alpha)/(1-s)$ in the last equality. To obtain the  sharpness of~\eqref{ineq:mart} we take $\alpha\to (p-1)/p ^+$ in the case $1<p\leq 2$ or take $\alpha=\alpha_p$ (see~Lemma~\ref{lem:alpha-a})  in the case $p>2$.
\end{proof}

\begin{remark}
In the above example $f_n$ converges a.s. to a random variable $f_{\infty}$, but $\EE |f_{\infty}|^p = +\infty$. In the case $1<p\leq 2$ one can consider $\alpha > (p-1)/p$ instead of $\alpha< (p-1)/p$ to obtain an example in which $\EE |f_{\infty}|^p<\infty$.
\end{remark}

\subsection{Relation to Theorem~\ref{thm:main} for $m=0$}
  \label{sec:proof-main}

As announced  before, the martingale inequality implies the main result for $m=0$ and $\lambda=1$ in the special case of nonincreasing functions.

\begin{proof}[Proof of Corollary~\ref{cor:main-monotone}]
First note that a standard approximation arguments yields a continuous time version of \eqref{ineq:mart}: for any martingale $(X_t)_{t\geq 0}$ with right-continuous trajectories, we have
\begin{equation*}
\|X_t-\sup_{0\leq s \leq t} X_s\|_p \leq \Cp \|X_t\|_p.
\end{equation*}

Let $f\in\Lpseg$ be a nonincreasing function. On the probability space $[0,1]$, equipped with the Lebesgue measure, consider the filtration
\begin{equation*}
\mathcal{F}_t = \sigma\big([0,1-t], \mathcal{B}([1-t,1])\big), \quad t\in[0,1],
\end{equation*}
(i.e. the $\sigma$-algebra $\mathcal{F}_t$ is generated by the set $[0,1-t]$ and all Borel subsets of the interval $[1-t,1]$)
and the martingale $X_t =\EE(f|\mathcal{F}_t)$, $t\in[0,1]$.
Using the definition of the filtration, we get the explicit formula
\begin{equation*}
 X_t(\omega)=	
  \begin{cases}
    \frac{1}{1-t} \int_0^{1-t} f(s) ds & \text{if }  \omega\in[0, 1-t),\\
    f(\omega) & \text{if } \omega\in[1-t,1].
  \end{cases}
\end{equation*}
In particular, the martingale is right-continuous. Hence,
\begin{align*}
\MoveEqLeft\Cp^p \|f\|_{\Lpseg}^p = \Cp^p\EE |X_1|^p\geq \EE |X_1-\sup_{0\leq t \leq 1} X_t|^p\\
&=\int_0^1 \Big|f(\omega) - \sup_{0\leq t\leq 1}\Big\{\indbr{\omega < 1-t}
\frac{1}{1-t} \int_0^{1-t} f(s) ds + \indbr{\omega \geq 1-t}f(\omega) 
    \Big\} \Big|^p d\omega\\
&= \int_0^1 \big|f(\omega) -  \frac{1}{\omega}\int_0^{\omega} f(s) ds \big|^p d\omega = \|f-\opH f \|_{\Lpseg}^p,
\end{align*}
where the second last equality follows from the fact that $f$ is nonincreasing. A~rescaling argument (see proof of inequality~\eqref{ineq:main} in Section~\ref{sec:lambda-m}) yields $\Cp \|f\|_{\Lp}\geq \|f-\opH f \|_{\Lp}$ for nonincreasing functions $f\in\Lp$.
\end{proof}

Let us now briefly sketch of the proof of the main result for $m=0$, $\lambda=1$, and not necessarily nonincreasing functions.
The key idea is to use the tools from the proof of the martingale inequality proof, i.e. the special functions $V$ and $U$ from Section~\ref{sec:majorant}, rather than to apply it directly. 
The following lemma is needed (its proof is based on integration by parts and we shall prove a more general result later, cf.~Lemma~\ref{lem:by-parts-pm}).

\begin{lemma}\label{lem:by-parts}
If $1<p<\infty$ and $f:[0,1]\to\RR$ is continuous, then
\begin{equation*}
(p-1)\int_0^1 |\opH f(t)|^p dt \leq p\int_0^1 |\opH f(t)|^{p-2}\opH f(t) f(t) dt.
\end{equation*}
\end{lemma}

Instead of looking at the values of $V$, $U$ in the points $(f_n, \sup_{k\leq n} f_n)$, where $(f_n)_{n=0}^{\infty}$ is a martingale, we consider their values in the points $(f(t),\opHm[0]f(t))$, where $f:[0,1]\to\RR$ is a continuous function.
 We use Proposition~\ref{prop:majorant} with $x=f(t)$ and $y=\opH f(t)$, integrate over $t\in[0,1]$, and apply Lemma~\ref{lem:by-parts}, arriving at
\begin{multline*}
\int_0^1 |f(t)-\opH f(t)|^p dt - \Cp^p \int_0^1 |f(t)|^p dt\\
\leq -\frac{|1-\alpha_p|^{p-2}}{p-1} \int_0^1 |\opH f(t)|^{p-2} \opH f(t)\big( p f(t) - (p-1) \opH f(t)\big) dt \leq 0.
\end{multline*}
Standard approximation and a simple scaling argument gives 
$\|f-\opH f\|_{\Lp} \leq \Cp \| f\|_{\Lp}$. As for sharpness, it is enough to consider the functions
\begin{equation*}
f_{\alpha}(t) = \indbr{t\in[0,1)}+ \alpha t^{\alpha-1} \indbr{t\in[1,\infty)}, \quad \alpha<(p-1)/p.
\end{equation*}

It turns out that this approach can be adapted to work in a more general setting. This is done with all details in the next section.

\section{Proof of Theorem~\ref{thm:main}}
  \label{sec:lambda-m}

\subsection{Notation, preliminary results}

For $1<p<\infty$ and $m>-2(p-1)/p$, we denote
\begin{equation*}
\gpm = \frac{m}{2} + \frac{p-1}{p}.
\end{equation*}
Clearly, $\gpm$ is a positive number. Moreover, we have the following slight extension of~\cite[Proposition~5.1]{MR2595549}.

\begin{proposition}
\label{prop:boundness}
For $1<p<\infty$ and $m>-2(p-1)/p$, the formula
\begin{equation*}
\opHm f(t) = \frac{1}{t^{1+m/2}} \int_0^t f(s) s^{m/2} ds
\end{equation*}
defines a bounded operator on the space $\Lp$.
Moreover,
\begin{equation*}
\opnorm{\opHm}{\Lp} = \gpm^{-1}.
\end{equation*}
\end{proposition}

\begin{proof} If $f\in\Lp$ is bounded, then the function $\opHm f$ is well defined (since $m/2> -1$), and the Minkowski integral inequality yields
\begin{align*}
\| \opHm f \|_p
&= \Big( \int_0^\infty \Big|\frac{1}{t^{1+m/2}} \int_0^t f(s) s^{m/2}ds\Big|^p dt \Big)^{1/p}\\
&= \Big( \int_0^\infty \Big|\int_0^1 f(ut) u^{m/2}du \Big|^p dt \Big)^{1/p}\\
&\leq  \int_0^1 \Big(\int_0^\infty |f(ut)|^p  dt \Big)^{1/p} u^{m/2}du
=\int_0^1 u^{m/2-1/p}\|f\|_p du  = \gpm^{-1}\|f\|_p.
\end{align*}
A standard density argument implies the claim. The family $t\mapsto t^{-1/p+\eps}\indbr{t\in[0,1]}$ extremizes the norm as $\eps\to0^+$.
\end{proof}

We define the set $\setpm$, the function $\cmpl:\setpm\to\RR$, and the constant $\Cpml$ by the formulas:
\begin{gather*}
\setpm = \{(\alpha,\beta)\in\RR^2 : \alpha < \gpm < \beta \},\\
\cmpl(\alpha,\beta) =\Big( \frac{(\beta-\gpm)|\alpha-\lambda|^p + (\gpm-\alpha)|\beta-\lambda|^p}{(\beta-\gpm)|\alpha|^p + (\gpm-\alpha)|\beta|^p}\Big)^{1/p},\quad (\alpha,\beta)\in\setpm,\\
\Cpml = \sup\{ \cmpl(\alpha,\beta) : \  (\alpha,\beta)\in \setpm \}.
\end{gather*}
Recall that our aim is to show that $\Cpml=\opnorm{\opId-\lambda\opHm}{\Lp}$. 
The first lemma shows that $\Cpml$ is a lower bound for the 
 norm of the operator $\opId - \lambda \opHm$. It also serves as a proof that $\Cpml$ is finite.

\begin{lemma}\label{lem:sharp}
If $1<p<\infty$, $m>-2(p-1)/p$, $\lambda\in\RR$, then
\begin{equation*}
\opnorm{\opId - \lambda \opHm}{\Lp}\geq \Cpml.
\end{equation*}
\end{lemma}

\begin{proof}
Fix $1<p<\infty$ and $m>-2(p-1)/p$.
For $\alpha<\gpm<\beta$, consider the function
\begin{equation}\label{eq:f-ab}
f_{\alpha,\beta}(t) = \beta t^{\beta-\gpm - 1/p} \indbr{t\in[0,1)}+\alpha t^{\alpha-\gpm - 1/p}\indbr{t\in[1,\infty)}, 
\end{equation}
which clearly belongs to the space $\Lp$. We have
\begin{equation*}
\opHm f_{\alpha,\beta}(t) = t^{\beta-\gpm - 1/p} \indbr{t\in[0,1)}+ t^{\alpha-\gpm - 1/p}\indbr{t\in[1,\infty)}
\end{equation*}
and
\begin{align*}
\|f_{\alpha,\beta} - \lambda\opHm f_{\alpha,\beta} \|_p^p &= \frac{|\beta-\lambda|^p}{p(\beta-\gpm)} -\frac{|\alpha-\lambda|^p}{p(\alpha-\gpm)},\\
\| f_{\alpha,\beta} \|_p^p &=  \frac{|\beta|^p}{p(\beta-\gpm)} -\frac{|\alpha|^p}{p(\alpha-\gpm)}.
\end{align*}
Thus we see that $\opnorm{\opId - \lambda \opHm}{\Lp}\geq \Cpml$.
\end{proof}

The next lemma summarizes further observations about the constant $\Cpml$.

\begin{lemma}\label{lem:Cpm}
If $1<p<\infty$, $m>-2(p-1)/p$, and $\lambda\in\RR$, then
\begin{equation*}
\Cpml \geq \max\{ |1-\lambda\gpm^{-1}|, 1\}.
\end{equation*}
Also, if the above inequality is strict, then the supremum in the definition of $\Cpml$ is attained at some point of the set $\setpm$. Moreover, $\Cpml > 1$ unless $\lambda=0$ or $p=2$.
\end{lemma}

\begin{proof}
Throughout the proof we consider only $(\alpha,\beta)\in\setpm$.
For $\alpha\neq0$, we can write the function $\cmpl^p$ as a convex combination:
\begin{equation}\label{eq:conv-comb}
\cmpl^p(\alpha,\beta) = w_1(\alpha,\beta)\cdot|1-\lambda/\alpha|^p + w_2(\alpha,\beta)\cdot |1-\lambda/\beta|^p,
\end{equation}
where
\begin{align*}
w_1(\alpha,\beta) = \frac{(\beta-\gpm)|\alpha|^p}{(\beta-\gpm)|\alpha|^p + (\gpm-\alpha)|\beta|^p}, \\
w_2(\alpha,\beta) = \frac{ (\gpm-\alpha)|\beta|^p}{(\beta-\gpm)|\alpha|^p + (\gpm-\alpha)|\beta|^p}.
\end{align*}
Using~\eqref{eq:conv-comb} we see that
\begin{equation*}
\lim_{\alpha,\beta} \cmpl^p(\alpha,\beta) =
   \begin{cases}
   1 & \text{if } \alpha\to-\infty \text{ and } \beta\to\infty,\\
   |1-\lambda\gpm^{-1}|^p & \text{if } \alpha\to\gpm^- \text{ and }  \beta\to\gpm^+,
   \end{cases}
\end{equation*}
which implies the first part of the assertion.

We now claim, that if $(\alpha,\beta)\to\partial\setpm$ or $\alpha^2+\beta^2\to\infty$, then
\begin{equation}\label{eq:boundary}
\limsup_{\alpha,\beta} \cmpl^p(\alpha,\beta) \leq \max\{1,|1-\lambda\gpm^{-1}|^p \}.
\end{equation}
It follows from~\eqref{eq:conv-comb} that~\eqref{eq:boundary} holds if $\alpha\to-\infty$ and $\beta\to\gpm^+$, or $\alpha\to\gpm^-$ and $\beta\to\infty$. If on the other hand $\alpha\to-\infty$ or $\alpha\to\gpm^-$, and $\beta\to \beta_{\infty}\in(\gpm,\infty)$, then $w_2(\alpha,\beta)\to 0$, and consequently
\begin{equation*}
\lim_{\alpha,\beta} \cmpl^p(\alpha,\beta) \in\{ 1, |1-\lambda\gpm^{-1}|\}.
\end{equation*}
Similarly,~\eqref{eq:boundary} also holds if $\alpha\to \alpha_{\infty}\in(-\infty, \gpm)\setminus\{0\}$ and $\beta\to \gpm^+$ or $\beta\to\infty$ (because then $w_1(\alpha,\beta)\to 0$). Finally, if $\alpha\to 0$ and $\beta\to \gpm^+$ or $\beta\to\infty$, then
\begin{multline*}
\limsup_{\alpha,\beta} \cmpl^p(\alpha,\beta) \leq \lim_{\alpha,\beta} \frac{(\beta-\gpm)|\alpha-\lambda|^p + (\gpm-\alpha)|\beta-\lambda|^p}{ (\gpm-\alpha)|\beta|^p}\\
=\max\{1, |1-\lambda\gpm^{-1}|\}.
\end{multline*}
These observations imply that if the supremum in the definition of $\Cpml$ is strictly greater than $\max\{1,|1-\lambda\gpm^{-1}| \}$, then it is attained at some point of the set $\setpm$.

The last part of the assertion clearly holds if $\lambda<0$. Assume henceforth that $\lambda>0$ and $p\neq 2$. Choose $A>0$ and $B >\max\{ \lambda- \gpm,0\}$ so that the inequality
\begin{equation*}
A^{p-1} B > (\gpm + A)(B+\gpm)^{p-1}
\end{equation*}
is satisfied (i.e. pick A sufficiently large if $p>2$ or $B$ sufficiently large if $1<p<2$). Since
\begin{gather*}
(A+\lambda)^p - A^p  \geq p\lambda A^{p-1},\\
 p\lambda(B+\gpm)^{p-1} \geq(\gpm+B)^p - (\gpm+B-\lambda)^p,
\end{gather*}
such a choice of $A,B$ implies that
\begin{equation*}
((A+\lambda)^p - A^p ) B > (\gpm+A) ((\gpm+B)^p - (\gpm+B-\lambda)^p),
\end{equation*}
which is equivalent to $\cmpl(-A, \gpm+B) >1$.
\end{proof}

\begin{remark}\label{rem:inv}
We can see that $\opnorm{\opId - \lambda \opHm}{\Lp}\geq 1$ in another way. Indeed, for $f_n(t) = \indbr{t\in[n,n+1)}$ have $\|f_n\|_p=1$, but
\begin{align*}
\MoveEqLeft[1]
\| \opHm f_n\|_p^p \\
&=\int_0^{\infty} \Big| \frac{t^{1+m/2} - n^{1+m/2}}{(1+m/2)t^{1+m/2}}\indbr{t\in[n,n+1)} +  \frac{(n+1)^{1+m/2} - n^{1+m/2}}{(1+m/2)t^{1+m/2}}\indbr{t\in[n+1,\infty)}\Big|^p dt\\
&\leq \int_n^{\infty} \Big(\frac{(n+1)^{1+m/2} - n^{1+m/2}}{(1+m/2)t^{1+m/2}}\Big)^p dt\\
&= \frac{((n+1)^{1+m/2} - n^{1+m/2})^p}{(1+m/2)^p (p+pm/2 -1) n^{p+pm/2 -1}} \xrightarrow[n\to\infty]{} 0.
\end{align*}
Hence the operator $\lambda\opHm:\Lp\to\Lp$ is not invertible, and consequently we cannot have $\opnorm{\opId - \lambda \opHm}{\Lp}< 1$.
\end{remark}

\subsection{Key tools}

The first result generalizes Lemma~\ref{lem:by-parts} presented above without proof.

\begin{lemma}\label{lem:by-parts-pm}
If $1<p<\infty$,  $m>-2(p-1)/p$, and $f:[0,1]\to\RR$ is continuous, then
\begin{equation*}
(p(1+m/2)-1)\int_0^1 |\opHm f(t)|^p dt \leq p\int_0^1 |\opHm f(t)|^{p-2}\opHm f(t) f(t) dt.
\end{equation*}
\end{lemma}

\begin{proof}
Define $F(t) = \int_0^t f(s)s^{m/2}ds$. Since $f$ is continuous, we have $F'(t) = f(t)t^{m/2}$ (in particular, $\opHm f(t) = F(t)/t^{1+m/2}\to f(0)/(1+m/2)$ as $t\to 0^+$). Hence integration by parts yields
\begin{align*}
\MoveEqLeft[1]
(p(1+m/2)-1)\int_0^1 |\opHm f(t)|^p dt = (p(1+m/2)-1)\int_0^1 t^{-p(1+m/2)}|F(t)|^p dt\\
&= \big[-t^{-p(1+m/2)+1} |F(t)|^p\big]_{0}^1 + p\int_0^1 t^{-p(1+m/2)+1}|F(t)|^{p-2}F(t)f(t)t^{m/2} dt\\
&= - |F(1)|^p + \lim_{t\to 0^+} t |F(t)/t^{1+m/2}|^p
+ p\int_0^1 |\opHm f(t)|^{p-2} \opHm f(t) f(t) dt\\
&= - |F(1)|^p + p\int_0^1 |\opHm f(t)|^{p-2} \opHm f(t) f(t) dt.
\end{align*}
This implies the assertion of the lemma.
\end{proof}

Moreover, the following elementary lemma is useful for us.

\begin{lemma}\label{lem:3con-pm}
Suppose that $v:\RR\to\RR$ is continuously differentiable and strictly concave on $(-\infty,a)$, strictly convex on $(a,b)$, and strictly concave on $(b,\infty)$ for some  $a,b\in\RR$. Let $u:\RR\to\RR$ be an affine function tangent to $v$ at two points. Then $v(x)\leq u(x)$ for $x\in\RR$.
\end{lemma}

\begin{proof} Denote  $c=u'(x)$, $x\in\RR$. There exist $\alpha<\beta$, such that $v(\alpha)=u(\alpha)$, $v(\beta) =u(\beta)$, and $v'(\alpha) = v'(\beta) = c$. By Rolle's theorem applied to the function $v-u$, there exists $\gamma\in(\alpha,\beta)$ such that $v'(\gamma)=c$. Since the function $v'$ attains every value at most thrice, we conclude that $\alpha\in(-\infty,a]$, $\gamma\in(a,b)$, and $\beta\in[b,\infty)$. The assertion follows from the assumption about concavity (respectively convexity) of $v$ on those intervals.
\end{proof}

Finally, we have the following analog and generalization of Proposition~\ref{prop:majorant}.

\begin{proposition}\label{prop:majorant-pm}
If $1<p<\infty$, $p\neq 2$,  $m>-2(p-1)/p$,  and $\lambda>0$, then there exists a positive constant $\Dpml$, such that the inequality
\begin{equation*}
|x-\lambda y|^p - \Cpml^p|x|^p \leq -\Dpml|y|^{p-2}y(x-\gpm y)
\end{equation*}
holds for all $x,y\in\RR$.
\end{proposition}

\begin{proof} We shall consider two cases. As will be clear from the proof (and the following results) they correspond to the situation when the supremum in the definition of $\Cpml$ is equal to $|\lambda\gpm^{-1}-1|$, and the situation when the supremum in definition  of $\Cpml$ is attained in the interior of the set $\setpm$ (cf. Lemma~\ref{lem:Cpm}).

For the first case, we assume that $\lambda > 2\gpm$ and the inequality
\begin{equation}\label{ineq:ban-jan}
|x-\lambda|^p - (\lambda\gpm^{-1}-1)^p|x|^p \leq -p(\lambda-\gpm)^{p-1}\lambda \gpm^{-1}(x-\gpm)
\end{equation}
holds for all $x\in\RR$. Define
\begin{align*}
V(x,y) &= |x-\lambda y|^p - (\lambda\gpm^{-1}-1)^p|x|^p,\\
U(x,y) &= -p(\lambda-\gpm)^{p-1}\lambda\gpm^{-1}|y|^{p-2}y(x-\gpm y).
\end{align*}
The inequality $V(x,y)\leq U(x,y)$ holds for $y=0$ (since $\lambda\gpm^{-1}-1 > 1$) and for $y=1$ (due to~\eqref{ineq:ban-jan}), and hence by homogeneity for all $x,y\in\RR$.  Hence the assertion is satisfied with $\Dpml = p(\lambda-\gpm)^{p-1}\lambda\gpm^{-1}>0$ (and with $\lambda\gpm^{-1}-1$, which is not greater than $\Cpml$, in the place of $\Cpml$). This finishes the proof in the first case.

Let us now consider the second case: we have either $\lambda\in(0,2\gpm]$, or we have $\lambda > 2\gpm$, but the inequality~\eqref{ineq:ban-jan} does not hold for all $x\in\RR$. We claim that the supremum in the definition of the constant $\Cpml$ is attained at some point of the set $\setpm$. Indeed, if $\lambda\in(0,2\gpm]$, then $\max\{|1-\lambda\gpm^{-1}|,1 \} =1$ and Lemma~\ref{lem:Cpm} implies the claim. If on the other hand $\lambda>2\gpm$ and the inequality~\eqref{ineq:ban-jan} does not hold for every $x\in\RR$, then there exists some $x_0\in\RR$, such that
\begin{equation}\label{ineq:ban-jan-not}
|x_0-\lambda|^p - (\lambda\gpm^{-1}-1)^p|x_0|^p > -p(\lambda-\gpm)^{p-1}\lambda\gpm^{-1}(x_0-\gpm).
\end{equation}
Of course we cannot have $x_0 = \gpm$. Suppose first, that $x_0 > \gpm$. Since
\begin{equation*}
\lim_{x\to\gpm} \frac{|x-\lambda|^p - (\lambda\gpm^{-1}-1)^p|x|^p}{x-\gpm} = -p(\lambda-\gpm)^{p-1}\lambda\gpm^{-1},
\end{equation*}
we conclude from~\eqref{ineq:ban-jan-not}, that for some $(\alpha,\beta)\in\setpm$ we have
\begin{equation*}
|\beta-\lambda|^p - (\lambda\gpm^{-1}-1)^p|\beta|^p > \frac{|\alpha-\lambda|^p - (\lambda\gpm^{-1}-1)^p|\alpha|^p}{\alpha-\gpm} (\beta-\gpm)
\end{equation*}
(it suffices to take $\beta=x_0$ and $\alpha$ smaller than, but close to $\gpm$) or equivalently
\begin{equation*}
\frac{(\beta-\gpm)|\alpha-\lambda|^p + (\gpm-\alpha)|\beta-\lambda|^p}{(\beta-\gpm)|\alpha|^p + (\gpm-\alpha)|\beta|^p} > (\lambda\gpm^{-1}-1)^p.
\end{equation*}
We arrive at the same conclusion, if $x_0 < \gpm$ (it suffices to take $\alpha = x_0$ and $\beta$ greater than, but close to $\gpm$). This finishes the proof of the claim: in the second case we always have
\begin{equation*}
\Cpml^p = \cmpl^p(\apml,\bpml),
\end{equation*}
for some point $(\apml,\bpml)\in\setpm$ (of course $\apml$, $\bpml$ may depend of $p$, $m$, and $\lambda$; uniqueness is not important to us).

After denoting
\begin{align*}
K(\alpha,\beta) &= (\beta-\gpm)|\alpha-\lambda|^p + (\gpm-\alpha)|\beta-\lambda|^p,\\
L(\alpha,\beta) &=(\beta-\gpm)|\alpha|^p + (\gpm-\alpha)|\beta|^p,
\end{align*}
we can rewrite the condition
$\frac{\partial}{\partial\alpha} \cmpl^p(\apml,\bpml) = 0$
as
\begin{align}
\begin{split}
\label{eq:partial-zero}
\MoveEqLeft[4]
\big(p(\bpml-\gpm)|\apml-\lambda|^{p-2}(\apml-\lambda) -|\bpml-\lambda|^p\big) \cdot L(\apml,\bpml)\\
& - K(\apml,\bpml)\cdot\big(p(\bpml-\gpm)|\apml|^{p-2}\apml -|\bpml|^p\big) = 0.
\end{split}
\end{align}
The condition $\frac{\partial}{\partial\beta} \cmpl^p(\apml,\bpml) = 0$ implies a similar equation.

Define now
\begin{align*}
V(x,y) &= |x-\lambda y|^p - \Cpml^p|x|^p,\\
U(x,y) &= \frac{V(\bpml,1) - V(\apml,1)}{\bpml-\apml}|y|^{p-2}y(x - \gpm y).
\end{align*}
Using the fact that $\Cpml^p = K(\apml,\bpml)/L(\apml,\bpml)$, we see that
\begin{align*}
V(\apml,1) &=
\frac{(\gpm-\apml)\big( |\apml-\lambda|^p|\bpml|^p   -|\apml|^p|\bpml-\lambda|^p\big) }{L(\apml,\bpml)},\\
V(\bpml,1) &=
\frac{(\bpml-\gpm)\big( |\apml|^p|\bpml-\lambda|^p   - |\apml-\lambda|^p|\bpml|^p\big) }{L(\apml,\bpml)},
\end{align*}
and consequently
\begin{equation}\label{eq:delta-V}
\frac{V(\bpml,1) - V(\apml,1)}{\bpml-\apml} = \frac{|\apml|^p|\bpml-\lambda|^p   - |\apml-\lambda|^p|\bpml|^p }{L(\apml,\bpml)}.
\end{equation}
Hence $V(\apml,1)=U(\apml,1)$ and $V(\bpml,1)=U(\bpml,1)$.

On the other hand (we use~\eqref{eq:partial-zero} in the second equality, and the definitions of $K$ and $L$ in the third),
\begin{multline*}
V_x(\apml,1)
=
\frac{p |\apml-\lambda|^{p-2}(\apml-\lambda) L(\apml,\bpml) - K(\apml,\bpml)p |\apml|^{p-2}\apml }{L(\apml,\bpml)}\\
=
\frac{L(\apml,\bpml) |\bpml-\lambda|^p- K(\apml,\bpml)|\bpml|^p }{L(\apml,\bpml)(\bpml-\gpm)}
=
\frac{|\apml|^p|\bpml-\lambda|^p   - |\apml-\lambda|^p|\bpml|^p }{L(\apml,\bpml)}.
\end{multline*}
By~\eqref{eq:delta-V}, we conclude that $V_x(\apml,1)=U_x(\apml,1)$. Similarly, $V_x(\bpml,1)=U_x(\bpml,1)$ follows from $\frac{\partial}{\partial\beta} \cmpl^p(\apml,\bpml) = 0$.

Therefore, $U(\cdot,1):x\mapsto U(x,1)$ is tangent to $V(\cdot,1):x\mapsto V(x,1)$ at $x=\apml$ and $x=\bpml$. Hence Lemma~\ref{lem:3con-pm} implies that $V(x,1)\leq U(x,1)$ for any $x\in\RR$, and by homogeneity also $V(x,y)\leq U(x,y)$ for any $x,y\in\RR$ (for $y=0$ the inequality holds, since $\Cpml>1$).

Finally, let us notice that
\begin{equation*}
V_x(\bpml,1) = p(|\bpml-1|^{p-2}(\bpml-1) - \Cpml^p\bpml^{p-1}) \leq p \bpml^{p-1}(1- \Cpml^p)<0
\end{equation*}
(we used the fact that $\bpml>0$ and $\Cpml>1$) and hence the assertion is satisfied with
\begin{equation*}
\Dpml := U_x(\bpml,1) = V_x(\bpml,1) <0.
  \qedhere
\end{equation*}
\end{proof}

\subsection{Proof of the main result}
\label{sec:lambda-pm-proofs}

We start with the following observation (cf.~\cite{MR816744}).

\begin{lemma}\label{lem:pelcz}
Let $T:\Lp\to\Lp$ be a linear operator which maps real-valued functions to real-valued functions. If the inequality $\|Tf\|_p \leq C \|f\|_p$  holds for any real-valued function $f\in\Lp$, then it also holds (with the same constant) for any complex-valued function $f\in\Lp$.
\end{lemma}

\begin{proof}
 Suppose that $f = u+iv\in\Lp$, where $u$, $v$ are real-valued. Let $G_1$, $G_2$ be independent Gaussian random variables with mean zero and variance one. Using the fact that for $a_1,a_2\in\RR$ the random variable $a_1G_1+a_2G_2$ has the same distribution as $\sqrt{a_1^2+a_2^2}G_1$, we arrive at
\begin{multline*}
\| T f\|_p^p \EE |G_1|^p = \int_0^{\infty} |T(u)^2+T(v)^2|^{p/2}\EE |G_1|^p dt
= \int_0^{\infty} \EE |T(u) G_1+T(v) G_2|^{p} dt \\
=\EE\int_0^{\infty} |T( uG_1 + vG_2)|^{p}dt\leq C^p\EE\int_0^{\infty} |uG_1+ vG_2|^{p}dt = C^p\| f\|_p^p\EE |G_1|^p
\end{multline*}
(we have suppressed the dependence of the functions on the argument $t\in[0,\infty)$ in the notation).
This finishes the proof of the lemma.
\end{proof}

Henceforth we assume without loss of generality that all functions are real-valued. The proof of~Theorem~\ref{thm:main} is divided into three parts.

\begin{proof}[Proof of inequality~\eqref{ineq:main} for $\lambda>0$ and $p\neq 2$.] Fix $1<p<\infty$, $p\neq 2$, $m>2(p-1)/p$, $\lambda\in\RR$, and denote
\begin{align*}
V(x,y) &= |x-\lambda y|^p - \Cpml^p|x|^p,\\
U(x,y) &= -\Dpml|y|^{p-2}y(x-\gpm y),
\end{align*}
where $\Dpml$ is the positive number from Proposition~\ref{prop:majorant-pm}.

Let $f:[0,1]\to\RR$ be a continuous function. By Proposition~\ref{prop:majorant-pm}, for every $t\in[0,1]$ we have $V(f(t),\opHm f(t))\leq U(f(t),\opHm f(t))$. After integrating over the interval $[0,1]$ and applying Lemma~\ref{lem:by-parts-pm}, we arrive at
\begin{multline*}
\int_0^1 |f(t)-\lambda \opHm f(t)|^p dt - \Cpml^p \int_0^1 |f(t)|^p dt\\
\leq -\Dpml \int_0^1 |\opHm f(t)|^{p-2}\opHm f(t) \big( f(t) - \gpm \opHm f(t)\big) dt \leq 0.
\end{multline*}
A standard approximation argument gives $\|f - \lambda\opHm f\|_{\Lpseg} \leq \Cpml \| f\|_{\Lpseg}$ for $f\in\Lpseg$. If $f\in\Lp$, then
\begin{multline*}
\int_0^n |f(t)-\lambda\opHm f(t)|^p dt = n \int_0^1 |f(nt)-\lambda\opHm (f(n\cdot)) (t)|^p dt \\
\leq n\Cpml^p \int_0^1 |f(nt)|^p dt = \Cpml^p \int_0^n |f(t)|^p dt,
\end{multline*}
and it suffices to take $n\to\infty$ to arrive at $\|f - \lambda \opHm f\|_{\Lp} \leq \Cpml \| f\|_{\Lp}$. This ends the proof of inequality~\eqref{ineq:main}.
\end{proof}

\begin{proof}[Proof of inequality~\eqref{ineq:main} for $\lambda>0$ and $p=2$] By an argument similar to that for $p\neq 2$, we show that $\opnorm{\opId- \lambda \opHm}{\Lp[2]}\leq \lambda\gpm[p=2,m]^{-1}-1$ for $\lambda> 2\gpm[p=2,m] = 1+m$. Indeed, we only need to use the functions
\begin{align*}
V(x,y) &= (x-\lambda y)^2 - (\lambda\gpm[p=2,m]^{-1}-1)^2x^2,\\
U(x,y) &= -2(\lambda-\gpm[p=2,m])\lambda\gpm[p=2,m]^{-1}y(x- \gpm[p=2,m] y),
\end{align*}
for which checking the majorization is straightforward ($x\mapsto U(x,1)$ is the tangent to the concave function $x\mapsto V(x,1)$ at $x=\gpm[p=2,m]$). We conclude that $\opnorm{\opId- \lambda \opHm}{\Lp[2]} = \lambda\gpm[p=2,m]^{-1}-1 = \Cpml[p=2,m,\lambda]$ for $\lambda>1+m$.

Moreover, the $L^2$-norm of the operator $\opId- \lambda \opHm$ is clearly a convex function of the variable $\lambda$:
\begin{align*}
\MoveEqLeft[2]
\opnorm{\opId- (s\lambda_1 +(1-s)\lambda_2)\opHm}{\Lp[2]}\\
&=
\opnorm{s (\opId- \lambda_1 \opHm) + (1-s)(\opId- \lambda_2 \opHm) }{\Lp[2]}\\
&\leq 
s\opnorm{ \opId- \lambda_1 \opHm}{\Lp[2]}
+(1-s)\opnorm{\opId- \lambda_2 \opHm }{\Lp[2]}.
\end{align*}
Since this norm is equal to $1$ for $\lambda=0$, tends to $1$ as $\lambda\to (1+m)^+$, and is always at least $1$ (by Lemmas~\ref{lem:sharp} and~\ref{lem:Cpm}, or Remark~\ref{rem:inv}), we conclude that $\opnorm{\opId- \lambda \opHm}{\Lp[2]}=1$ for $\lambda\in[0,1+m]$.
\end{proof}

\begin{remark}
The operator $\opId - (1+m)\opHm$ is an isometry on $\Lp[2]$ (since the Beurling-Ahlfors transform is an $L^2$-isometry; see~\cite[Theorem~1.1]{MR2595549}).
\end{remark}

\begin{proof}[Proof of inequality~\eqref{ineq:main} for $\lambda\leq 0$.]
The triangle inequality and Lemma~\ref{lem:Cpm} yield
\begin{equation*}
\opnorm{\opId - \lambda\opHm}{\Lp} \leq  |\lambda|\gpm^{-1} +1 \leq \Cpml.
\end{equation*}
Moreover, the opposite inequality is also true by Lemma~\ref{lem:sharp}, so we in fact have equalities above.
\end{proof}

Sharpness of inequality~\eqref{ineq:main} follows from Lemma~\ref{lem:sharp}. In order to complete the proof of Theorem~\ref{thm:main}, we have to explain why $\Cpml[p,m=0,\lambda=1]^p$ is equal to
\begin{equation*}
\Cp^p =\sup \{ \cmpl[p,m=0,\lambda=1](\alpha,1) : \alpha < (p-1)/p \}  =
  \begin{cases}
  \frac{1}{(p-1)^p} & \text{if } 1<p\leq 2,\\
  \frac{(1+|\alpha_p|)^{p-2}}{p-1} & \text{if } p> 2
  \end{cases}
\end{equation*}
(see Section~\ref{sec:mart} for the definition of $\alpha_p$ and details).

Rather than to check directly that the supremum in the definition of $\Cpml[p,0,1]$ is attained in the point $(\alpha_p, 1)$ we refer to results obtained in Section~\ref{sec:mart}. For $m=0$ and $\lambda=1$, the above proof of inequality~\eqref{ineq:main}, can be repeated with the functions  $V$, $U$ being defined like in Subsection~\ref{sec:majorant} (of course, instead of using Proposition~\ref{prop:majorant-pm}, we use Proposition~\ref{prop:majorant}). This gives us $\|f - \opHm[0] f\|_{\Lp} \leq \Cp \| f\|_{\Lp}$, but since clearly $\Cp \leq \Cpml[p,0,1]$, and the constant $\Cpml[p,0,1]$ is best possible in this inequality, we conclude that $\Cp = \Cpml[p,0,1]$.

\section*{Acknowledgments}

I thank Adam Os\c{e}kowski for suggesting the study of the Hardy operator by the means of maximal martingale inequalities, as well as for his invaluable help, and many fruitful and inspiring discussions.

I also thank Tomasz Tkocz, whose nicely written bachelor thesis helped me to understand the connection between  Doob's inequality and the Hardy operator, when I was first learning  about them.

\bibliographystyle{amsplain}
\bibliography{radial-BA}

\def\cprime{$'$} \def\polhk#1{\setbox0=\hbox{#1}{\ooalign{\hidewidth
  \lower1.5ex\hbox{`}\hidewidth\crcr\unhbox0}}} \def\cprime{$'$}
  \def\cprime{$'$} \def\polhk#1{\setbox0=\hbox{#1}{\ooalign{\hidewidth
  \lower1.5ex\hbox{`}\hidewidth\crcr\unhbox0}}}
\providecommand{\bysame}{\leavevmode\hbox to3em{\hrulefill}\thinspace}
\providecommand{\MR}{\relax\ifhmode\unskip\space\fi MR }
\providecommand{\MRhref}[2]{%
  \href{http://www.ams.org/mathscinet-getitem?mr=#1}{#2}
}
\providecommand{\href}[2]{#2}
\begin{thebibliography}{10}

\bibitem{MR1294669}
Kari Astala, \emph{Area distortion of quasiconformal mappings}, Acta Math.
  \textbf{173} (1994), no.~1, 37--60. \MR{1294669 (95m:30028b)}

\bibitem{MR2001941}
R.~Ba{\~n}uelos and P.~J. M{\'e}ndez-Hern{\'a}ndez, \emph{Space-time {B}rownian
  motion and the {B}eurling-{A}hlfors transform}, Indiana Univ. Math. J.
  \textbf{52} (2003), no.~4, 981--990. \MR{2001941 (2004h:60067)}

\bibitem{MR2386238}
Rodrigo Ba{\~n}uelos and Prabhu Janakiraman, \emph{{$L^p$}-bounds for the
  {B}eurling-{A}hlfors transform}, Trans. Amer. Math. Soc. \textbf{360} (2008),
  no.~7, 3603--3612. \MR{2386238 (2009d:42032)}

\bibitem{MR2595549}
\bysame, \emph{On the weak-type constant of the {B}eurling-{A}hlfors
  transform}, Michigan Math. J. \textbf{58} (2009), no.~2, 459--477.
  \MR{2595549 (2011a:47096)}

\bibitem{MR3018958}
Rodrigo Ba{\~n}uelos and Adam Os{\c{e}}kowski, \emph{Sharp inequalities for the
  {B}eurling-{A}hlfors transform on radial functions}, Duke Math. J.
  \textbf{162} (2013), no.~2, 417--434. \MR{3018958}

\bibitem{MR3189475}
\bysame, \emph{On the operator {$\Lambda^*$} and the {B}eurling-{A}hlfors
  transform on radial functions}, Michigan Math. J. \textbf{63} (2014), no.~1,
  213--221. \MR{3189475}

\bibitem{MR1370109}
Rodrigo Ba{\~n}uelos and Gang Wang, \emph{Sharp inequalities for martingales
  with applications to the {B}eurling-{A}hlfors and {R}iesz transforms}, Duke
  Math. J. \textbf{80} (1995), no.~3, 575--600. \MR{1370109 (96k:60108)}

\bibitem{MR3047469}
Alexander Borichev, Prabhu Janakiraman, and Alexander Volberg,
  \emph{Subordination by conformal martingales in {$L^p$} and zeros of
  {L}aguerre polynomials}, Duke Math. J. \textbf{162} (2013), no.~5, 889--924.
  \MR{3047469}

\bibitem{MR1108183}
Donald~L. Burkholder, \emph{Explorations in martingale theory and its
  applications}, \'{E}cole d'\'{E}t\'e de {P}robabilit\'es de {S}aint-{F}lour
  {XIX}---1989, Lecture Notes in Math., vol. 1464, Springer, Berlin, 1991,
  pp.~1--66. \MR{1108183 (92m:60037)}

\bibitem{MR577984}
S.~D. Chatterji, \emph{Some comments on the maximal inequality in martingale
  theory}, Measure theory, {O}berwolfach 1979 ({P}roc. {C}onf., {O}berwolfach,
  1979), Lecture Notes in Math., vol. 794, Springer, Berlin, 1980,
  pp.~361--364. \MR{577984 (82a:60068)}

\bibitem{MR2164413}
Oliver Dragi{\v{c}}evi{\'c} and Alexander Volberg, \emph{Bellman function,
  {L}ittlewood-{P}aley estimates and asymptotics for the {A}hlfors-{B}eurling
  operator in {$L^p(\Bbb C)$}}, Indiana Univ. Math. J. \textbf{54} (2005),
  no.~4, 971--995. \MR{2164413 (2006i:30025)}

\bibitem{MR2677626}
James~T. Gill, \emph{On the {B}eurling-{A}hlfors transform's weak-type
  constant}, Michigan Math. J. \textbf{59} (2010), no.~2, 353--363. \MR{2677626
  (2012b:42022)}

\bibitem{MR719167}
T.~Iwaniec, \emph{Extremal inequalities in {S}obolev spaces and quasiconformal
  mappings}, Z. Anal. Anwendungen \textbf{1} (1982), no.~6, 1--16. \MR{719167
  (85g:30027)}

\bibitem{MR0181748}
Olli Lehto, \emph{Remarks on the integrability of the derivatives of
  quasiconformal mappings}, Ann. Acad. Sci. Fenn. Ser. A I No. \textbf{371}
  (1965), 8. \MR{0181748 (31 \#5975)}

\bibitem{MR2964297}
Adam Os{\polhk{e}}kowski, \emph{Sharp martingale and semimartingale
  inequalities}, Instytut Matematyczny Polskiej Akademii Nauk. Monografie
  Matematyczne (New Series) [Mathematics Institute of the Polish Academy of
  Sciences. Mathematical Monographs (New Series)], vol.~72,
  Birkh\"auser/Springer Basel AG, Basel, 2012. \MR{2964297}

\bibitem{MR3119338}
Adam Os{\c{e}}kowski, \emph{Sharp logarithmic bounds for {B}eurling-{A}hlfors
  operator restricted to the class of radial functions}, Mediterr. J. Math.
  \textbf{10} (2013), no.~4, 1883--1894. \MR{3119338}

\bibitem{MR816744}
A.~Pe{\l}czy{\'n}ski, \emph{Norms of classical operators in function spaces},
  Ast\'erisque (1985), no.~131, 137--162, Colloquium in honor of Laurent
  Schwartz, Vol. 1 (Palaiseau, 1983). \MR{816744 (87b:47036)}

\bibitem{MR2068982}
A.~Vol{\cprime}berg and F.~Nazarov, \emph{Heat extension of the {B}eurling
  operator and estimates for its norm}, Algebra i Analiz \textbf{15} (2003),
  no.~4, 142--158. \MR{2068982 (2005f:30042)}

\bibitem{volberg}
Alexander Volberg, \emph{{A}hlfors-{B}eurling operator on radial functions},
  Preprint (2012), {\tt arXiv:1203.2291}.

\end{thebibliography}

\end{document}